\documentclass[a4paper,12pt,reqno]{amsart}
\usepackage{amssymb}

\theoremstyle{theorem}
\newtheorem{theorem}{Theorem}[section]
\newtheorem{prop}[theorem]{Proposition}
\newtheorem{lemma}[theorem]{Lemma}

\DeclareMathOperator*{\aplim}{ap\,lim}

\begin{document}

\title{Non approximate derivability of the Takagi function}

\author{Juan Ferrera}
\address{IMI, Departamento de An{\'a}lisis Matem{\'a}tico y Matem{\'a}tica Aplicada, Facultad Ciencias Ma\-te\-m{\'a}ticas, Universidad Complutense, 28040, Madrid, Spain}
\email{ferrera@ucm.es}

\author{Javier G\'omez Gil}
\address{Departamento de An{\'a}lisis Matem{\'a}tico y Matem{\'a}tica Aplicada, Facultad Ciencias Ma\-te\-m{\'a}ticas, Universidad Complutense, 28040, Madrid, Spain}
\email{gomezgil@mat.ucm.es}



\begin{abstract}
The Takagi function is a classical example of a continuous nowhere differentiable
function. In this paper we prove that it is nowhere approximately  derivable.
\end{abstract}

\maketitle

\section{Introduction}

The Takagi function is probably the easiest example of
a continuous nowhere derivable function, it was introduced in 1903 (see \cite{Takagi}),  and ever since, it has caught the interest
of mathematicians, as a matter of fact it was often rediscovered, for instance, in 1930, using base ten instead of
base two by Van der Waerden (see \cite{VW}). For an extensive information about this function, see the surveys
\cite{AK} and \cite{Lagarias}.

We say that a function $f$ satisfies the $C^1$ Lusin property if for every $\varepsilon >0$ there exists
a $C^1$ function $g$ such that the set $\{ x: f(x)\not= g(x)\}$ has Lebesgue measure less than $\varepsilon$.
It is known that the Takagi's function does not satisfy the $C^1$ Lusin property, more precisely, that it
agrees with no $C^1$ function on any set of positive measure (see \cite{BK}). In particular, this implies that
the Takagi function does not satisfy the Lipschitz property (see \cite{Marcin}), however it is ``almost'' Lipschitz in the
sense  that it is $\alpha$-H\"{o}lder continuous for every $0<\alpha <1$ (see \cite{ShS}).

It is known that every function with non empty Fréchet subdifferential a.e. satisfies the $C^1$ Lusin property (see \cite{AFGGG}),
hence the Takagi function cannot be Fréchet subdifferentiable a.e., as a matter of fact it is known that it has non empty subdifferential
at $x$ if and only if $x$ is a dyadic number (see  \cite{FerreraGomezGil} or \cite{GoraStern}).
 
The Lusin properties are closely related with the approximate differentiability, for instance a one dimensional function
has the $C^1$ Lusin property if and only if it is approximately derivable a.e., for this result see \cite{Whitney} or \cite{LiuTai}
for similar results involving funcions on $\mathbb{R}^n$.

From all these results we may deduce that the Takagi function is not approximately derivable a.e., in other words:
there exists a positive measure set such that the Takagi function is not approximately derivable at any point of that set.
The aim of this paper is to give a direct, and relatively elementary, proof of a stronger result, namely that the
Takagi function is nowhere approximately derivable. Although it is known that almost every, in the sense of 
category, continuous function 
$f:\mathbb{R}\to \mathbb{R}$ is nowhere approximately derivable, see \cite{J}, it is not easy to provide examples of such functions, 
see for instance the examples that appear in \cite{J} or
\cite{K}, clearly the functions that they define are not as elementary as the Takagi
function. On the other hand the relevance of the Takagi function worth
the study of its approximate derivability.

\section{Results}
\label{sec:results}

We introduce some definitions and notation. $\mathcal{L}$ will denote the Lebesgue measure on $\mathbb{R}$.
For a function $f:\mathbb{R} \rightarrow \mathbb{R}$, $\ell$ is the
approximate limit of $f$ at $x$, $\aplim_{y\to x}f(y)=\ell$, if for every $\varepsilon >0$,
\begin{equation*}
\lim_{r\to 0^+}\frac{\mathcal{L}\bigl( (x-r,x+r)\cap \{ y\in \mathbb{R}: |f(y)-\ell|\geq \varepsilon \} \bigr)}{2r}=0.
\end{equation*}
We say that a function $f:\mathbb{R}\rightarrow \mathbb{R}$ is approximately derivable at $x$ if there
exists a real number $f'_{ap}(x)$, named approximate derivative of $f$ at $x$, such that
\begin{equation*}
 \aplim_{y\to x}\frac{f(y)-f(x)-f'_{ap}(x)(y-x)}{y-x}=0.
\end{equation*}

The following result, see \cite{K} page 139, is immediate:

\begin{prop}\label{sec:introduction-4}
 Let $f : \mathbb{R} \rightarrow \mathbb{R}$ be a function, let $x$ be a point of $\mathbb{R}$
and suppose that $f$ is approximately derivable at $x$ with approximate derivative $f'_{ap}(x)$.
Then, for
any real number $\alpha > f'_{ap}(x)$, we have
\begin{equation*}
\lim_{r\to 0^+} \frac{\mathcal{L}\left\{y \in (x - r, x + r) \setminus
  \{x\} :\; \frac{f (y) - f (x)}{y - x} \geq  \alpha \right\}}{2r} = 0.
\end{equation*}
Similarly, for any real number $\beta< f'_{ap}(x)$, we have
\begin{equation*}
\lim_{r\to 0^+} \frac{\mathcal{L}\left\{y \in (x - r, x + r) \setminus
  \{x\} :\; \frac{f (y) - f (x)}{y - x} \leq  \beta \right\}}{2r} = 0.
\end{equation*}
\end{prop}

We proceed to define the Takagi function.
For $n\in \mathbb{N}$, $n\geq 0$, let $D_n=\left\{k2^{-n}:\; k\in
  \mathbb{Z}\right\} \subset \mathbb{R}$  and let  $D=\cup D_n$ be the set of all dyadic numbers.
 The Takagi function $T:\mathbb{R}\rightarrow
 \mathbb{R}$   is defined as
\begin{equation}
T(x)=\sum_{k=1}^{\infty} g_k(x)=\lim_n G_n(x)\label{eq:1}
\end{equation}
where $g_k(x)=\min (|x-y|:\, y\in D_k)$, and $G_n=g_1+\dots +g_n$.
The aim of this paper is to prove the following Theorem:
 \begin{theorem}
   \label{sec:introduction}
   The Takagi function is nowhere approximately derivable.
\end{theorem}

We split the proof in several Propositions.
First, we assume that $x\not\in D$, we have that
for all $n$ there exist $x_n,y_n\in D_n$ such that
$x\in (x_n,y_n)$ and $(x_n,y_n)\cap D_n=\varnothing$.

It is clear that for $k< n$ we have that
$(x_n,y_n)\subset (x_k,y_k)$ and
\begin{equation}\label{eq:3}
  g_k(x')-g_k(x) = g'_k(x)(x'-x)
\end{equation}
for all $x'\in[x_n,y_n]$. Note that $g_k$ is derivable at $x$
because $x\not\in D$, moreover $g'_k(x)=\pm 1$. We start with the following estimation:

\begin{lemma}\label{sec:introduction-1}
  Let $x\not\in D$. If $g'_n(x)=1$ then
  \begin{equation}
    \label{eq:2}
    \mathcal{L}\left\{y
  :\;  0<|x-y|<2^{-n}, \,
      \frac{T(y)-T(x)}{y-x}\leq G'_{n-1}(x)+\frac25\right\} \geq \frac1{2^{n+5}}.
  \end{equation}

Analogously, if   $g'_n(x)=-1$ then
  \begin{equation}
   \label{eq:4}
    \mathcal{L}\left\{y
  :\; 0<|x-y|<2^{-n}, \,
      \frac{T(y)-T(x)}{y-x}\geq G'_{n-1}(x)-\frac25\right\} \geq \frac1{2^{n+5}}.
  \end{equation}
\end{lemma}
\begin{proof}
  Let $x_n, y_n$ be as above, if
  $g'_n(x)=1$ then $x$ is closer to $x_n$ than to $y_n$, hence
  $y_n-x>2^{-n-1}$. For $0<t< 2^{-n-5}$  if
    $k=0,1,2,\dotsc$, as  $y_n\in
    D_n\subset D_{n+k}$ and $\sup |g_{n+k}|\leq  2^{-n-k-1}$, we have that
  \begin{equation*}
    g_{n+k}(y_n-t)\leq \min\left\{2^{-n-5}, 2^{-n-k-1}\right\}
  \end{equation*}
and then, by \eqref{eq:3},
  \begin{align*}
    T(y_n-t)-T(x)\leq & G_{n-1}(y_n-t)-G_{n-1}(x)+ \frac5{2^{n+5}}+ \sum_{k>
      4}\frac1{2^{n+k+1}} \\ \leq & G'_{n-1}(x)(y_n-t-x)+ \frac6{2^{n+5}}.
  \end{align*}
Therefore,
\begin{equation*}
  \frac{T(y_n-t)-T(x)}{y_n-t-x}\leq G'_{n-1}(x)+\frac2{5}.
\end{equation*}
since $y_n-t-x\geq 2^{-n-1}-2^{-n-5}=15\cdot 2^{-n-5}$.
 The case $g'_n(x)=-1$ is similar.
\end{proof}

\begin{prop}
  \label{sec:introduction-5}
  Let $x\not\in D$, if either $I=\liminf_n G'_n(x)$ or $S=\limsup_nG'_n(x)$,
 is finite then $T$ is not approximately derivable at $x$.
\end{prop}

\begin{proof}
Let us suppose that $I$ is finite and $T$ is approximately derivable at $x$.
  As $I$ is finite and $G'_n(x)\in \mathbb{Z}$ for all
    $n$, then  there exists $n_0$ such that $G'_n(x)\geq I$ for
  all $n\geq n_0$ and there exists a strictly  increasing subsequence
  $(n_k)_{k\geq 0}$ such that $G'_{n_k-1}(x)=I$ for all $k$. It is
  clear that  $g'_{n_k-1}(x)=-1$ and $g'_{n_k}(x)=+1$ for all $k>0$
  and hence, by Lemma \ref{sec:introduction-1},
  \begin{multline*}
 \limsup_{r\to 0^+} \frac{ \mathcal{L}\left\{y : \, 0<|y-x|<r, \,
      \frac{T(y)-T(x)}{y-x}\leq I+\frac25\right\}  }{2r} 
\\
\geq \limsup_k \frac{ \mathcal{L}\left\{y
  :\;  0<|y-x|<2^{-n_k}, \,
      \frac{T(y)-T(x)}{y-x}\leq G'_{n_k-1}(x)+\frac25\right\} }{2^{-n_k+1}}
\\
\geq \frac{1}{2^6}>0
\end{multline*}
and
\begin{multline*}
  \limsup_{r\to 0^+} \frac{ \mathcal{L}\left\{y : \, 0<|y-x|<r, \,
      \frac{T(y)-T(x)}{y-x}\geq I+1-\frac25\right\} }{2r} 
\\
\geq \limsup_k \frac{ \mathcal{L}\left\{y :\; 0<|y-x|<2^{-n_k+1}, \,
    \frac{T(y)-T(x)}{y-x}\geq G'_{n_k-2}(x)-\frac25\right\}
}{2^{-n_k+2}} 
\\
\geq \frac{1}{2^6}>0
\end{multline*}
since $I+1=G'_{n_k-2}(x)$. We deduce, by Proposition \ref{sec:introduction-4},
the following contradiction.

\begin{equation*}
    T'_{ap}(x) \leq I+\frac2{5} < I+1-\frac2{5}\leq  T'_{ap}(x).
  \end{equation*}
The case $S$ finite is similar.
\end{proof}

The alternative occurs when both limits are infinite. Then we also have that $T$ is not approximately derivable.

\begin{prop} \label{sec:introduction-6}
Let $x\not\in D$, if
$I=\liminf_n G'_n(x)$ and $S=\limsup_nG'_n(x)$ are infinite,
  then $T$ is not approximately derivable at $x$.
\end{prop}

\begin{proof}
  As the result is obvious if $I=S$, we will suppose that $I=-\infty$
  and $S=+\infty$. Then, there exists a subsequence $(n_k)_k$
  such that
  \begin{align*}
    & G'_{n_k}(x)< G'_{n_{k+1}}(x),& \quad  & g'_{n_k+1}(x)=-1;\quad&
  \end{align*}
for all $k$, and
\begin{equation*}
  G'_{n_k}(x)\xrightarrow[k]{}+\infty.
\end{equation*}

By Proposition \ref{sec:introduction-4}, if $T$ is approximately derivable at
$x$ and $\alpha > T'_{ap}(x)$ there exists $\delta>0$ such that
\begin{equation}\label{eq:5}
  \mathcal{L}\left\{y \in (x - r, x + r) \setminus
  \{x\} :\; \frac{T (y) - T (x)}{y - x} \geq  \alpha
\right\}<\frac1{32}\, r
\end{equation}
for $0<r<\delta$. But if $k$ is such that $G'_{n_k-1}(x)>\alpha+1$ and
$2^{-n_k}<\delta$ we have, by  Lemma \ref{sec:introduction-1}, that
\begin{equation*}
  \begin{split}
    \mathcal{L}\left\{y: 0<|y-x|<\frac1{2^{n_k}},
      \frac{T(y)-T(x)}{y-x}\geq G'_{n_k-1}(x)-\frac25\right\} \\
    \geq \frac1{2^{n_k+5}}=\frac1{32}\,\frac1{2^{n_k}}
  \end{split}
 \end{equation*}
that contradicts \eqref{eq:5}.
\end{proof}
It only remains to prove that $T$ is not approximately derivable at any dyadic number.

\begin{prop}\label{dimension1}
If $x\in D$, then $T$ is not approximately derivable at $x$.
\end{prop}

\begin{proof}
  If $x\in D$, let $n_0$ be the smaller integer such that $x\in D_{n_0+1}$. Let $n$ be such that $n>2n_0$. If $2|h|<2^{-n}$ then
$g_k(x+h)=|h|$ for every $k$, $n_0< k\leq n$. Thus
\begin{equation*}
  \begin{split}
    T(x+h)-T(x) \geq & \sum_{k=1}^{n_0}\bigl( g_k(x+h)-g_k(x)\bigr)
    +\sum_{k=n_0+1}^ng_k(x+h) 
    \\
    \geq & -n_0|h|+(n-n_0)|h|=\left(n-2n_0\right)|h|.
  \end{split}
\end{equation*}

This implies that
\begin{equation*}
  \mathcal{L}\left\{y: 0<|y-x|<\frac1{2^{n+1}},
 \frac{T(y)-T(x)}{y-x}\geq n-2n_0\right\}   \geq  \frac1{4}\,{2^{-n}}.
   \end{equation*}
It follows from Proposition \ref{sec:introduction-4} that $T$ is not approximately derivable at $x$.
 \end{proof}

Joining  Propositions \ref{sec:introduction-5}, \ref{sec:introduction-6},
and \ref{dimension1}, we have the proof of Theorem  \ref{sec:introduction}

\end{document}